\documentclass[letterpaper,10pt,conference]{ieeeconf}  

\IEEEoverridecommandlockouts                              
\usepackage{color}
\usepackage{amsmath}
\usepackage{amssymb}

\newtheorem{assumption}{Assumption}
\newtheorem{theorem}{Theorem}

\newtheorem{lemma}{Lemma}
\newtheorem{remark}{Remark}
\newtheorem{definition}{Definition}

\def\EE{\mathsf{E}}
\def\PP{\mathsf{P}}

\let\cal\mathcal

\newcommand{\spn}{{\rm spn}}

\title{\LARGE \bf
Threshold Structure of Optimal Policies in Restart POMDPs
}

\author{Konstantin Avrachenkov, Alexey Piunovskiy and Yi Zhang
\thanks{This work is supported by a grant from NSP SmartProfile and has been accepted to IEEE CDC 2026.}
\thanks{K. Avrachenkov is with Inria, France
        {\it k.avrachenkov@inria.fr}}%
\thanks{A. Piunovskiy is with Department of Mathematical Sciences, University of Liverpool, UK
        {\it piunov@liverpool.ac.uk}}%
\thanks{Y. Zhang is with School of Mathematics, University of Birmingham, UK
        {\it y.zhang.29@bham.ac.uk}}%
}

\begin{document}

\maketitle
\thispagestyle{empty}
\pagestyle{empty}

\begin{abstract}
We study a Restart POMDP (Partially Observable Markov Decision Process) on a general Borel state space, where the controller either lets the hidden state evolve unobserved or restarts the system and observes the new state. Exploiting a sufficient-statistic representation consisting of the last observed state and the elapsed time since restart, we reduce the problem to a fully observed MDP. Under a natural one-step cost deterioration condition, we prove that optimal policies have a threshold structure in the elapsed time for both the discounted and total undiscounted cost criteria. When the state space is partially ordered and the kernel is stochastically monotone, we further show that the optimal threshold is nonincreasing in the state. For the average cost criterion, under additional assumptions of geometric ergodicity and domination of the transient gain, we establish analogous threshold results via the vanishing discount approach, after showing the uniform boundedness of the optimal thresholds and relative value functions.
\end{abstract}

\section{Introduction}


Many controlled systems evolve for long periods without direct observation, while costly interventions such as inspection, reset, repair, or replacement reveal the current state and improve system and information quality. This situation arises in applications including machine maintenance and replacement \cite{BP96,GMA05,W79}, healthcare \cite{Metal20}, sensing/radars \cite{V16}, telecommunications \cite{JS12,LWZ11,NM08}, age of information \cite{H18,TM24} and web crawling \cite{AB16}. Motivated by such problems, we study a Restart POMDP with the general Borel state space and two actions: {\it no intervention}, under which the hidden state evolves according to a Markov kernel and remains unobserved, or {\it restart}, under which the system is reset according to a given distribution and the state is observed. The key tradeoff is between delaying restart to avoid intervention costs and restarting early to improve both the hidden state and the controller’s information. We investigate discounted, total undiscounted as well as average cost optimality and the structure of optimal policies.


As is standard in the POMDP literature, one may reformulate the problem as a belief-state MDP with complete observation. For the Restart POMDP, this reduction has an especially convenient form. The controller’s information is fully summarized by two sufficient statistics: the last observed state and the time elapsed since the most recent intervention.

Our first contribution is to establish, under a natural one-step cost deterioration condition, that the optimal policy for the discounted and total undiscounted cost criteria has a threshold structure in the elapsed time since the last restart.
In general, the optimal threshold depends on the last observed state of the system. When the underlying Borel state space is equipped with a partial order and the Markov kernel is stochastically monotone, we further show that the optimal threshold is nonincreasing with respect to this order. In the second part of the paper, we establish analogous structural results for the average-cost criterion under assumptions of geometric ergodicity and domination of the transient gain. The latter condition rules out the trivial case in which restarting is never optimal. Our analysis of the average cost problem relies on the vanishing discount approach together with uniform boundedness of the relative value function. Beyond their theoretical interest, these structural results can facilitate efficient numerical schemes and guide the design of structure-aware reinforcement learning algorithms by reducing the policy search space.

Let us review the related literature. Reference \cite{D63} was among the first to study an MDP model with restart for optimal machine replacement policies. In \cite{D63}, the state of a machine is observed at periodic epochs. Then, in \cite{W79}, noise-corrupted observations of the hidden system state were introduced into the model of \cite{D63}. Both \cite{D63} and \cite{W79} established monotonicity properties of optimal threshold policies for finite-state models, under suitable monotonicity assumptions on the underlying Markov kernel. Then, \cite{KK10} studied the effect of a non-stationary environment, also in the finite-state setting. There is a significant body of related work on Restless Multi-Armed Bandits with an active action that restarts the arm from a given distribution \cite{AM22,AM24,AB16,H18,JS12,LWZ11,Metal20,TM24,V16}. Most of these works, except \cite{AB16,H18,TM24}, assume that the arms evolve in finite-state spaces. In \cite{AB16,H18,TM24}, the state space is countable, but the state evolution is deterministic and has very particular forms.
Markov processes with restart in general Borel spaces and without control have been studied in \cite{APZ13} in continuous time and in \cite{APZ18} in discrete time, respectively. See \cite{EMS20} for a comprehensive review of Markov processes with restart without control. The current work provides a significant generalization of the structural results for optimal policies in Restart POMDP from finite-state spaces to general Borel state spaces. We note that the structure of the belief space in this work is very similar to that of the belief state in Partially Observable Restless Bandits \cite{AM24} and Intermittently Observable MDPs \cite{CL25,ADK25}. Finally, in the average cost setting, we establish the uniform boundedness of the optimal threshold, creating a connection to recent research on finite-memory feedback policies \cite{KY22}.

\section{Restart POMDP}

Let us consider the POMDP model $\{{\bf X},{\bf O}, {\bf A}, \hat{P},\hat{O},\hat{O}_0, l\},$  
where ${\bf X}$ is the Borel state space, endowed with the Borel $\sigma$-algebra ${\cal B}({\bf X})$,
${\bf O}:={\bf X}\cup\{*\}$ is the observation space, where the isolated point $*$ represents that there is no information about the current hidden state, ${\bf A}=\{n,r\} $ is the action space, where $n$ (or $r$) represents the action of no intervention (resp., restart), 
$\hat{P}$ is a stochastic kernel on ${\bf X}$ given ${\bf X}\times {\bf A}$ in the form 
\[\hat{P}(dx'|x,a)=\begin{cases}
P(dx'|x) &\mbox{ if $a=n$};\\
\nu(dx') & \mbox{ if $a=r$,}
\end{cases}
\]
where $P$ is a stochastic kernel on ${\bf X}$ given $ {\bf X}$ and $\nu$ is a given distribution on ${\bf X}$, $\hat{O}$ is the observation kernel, a stochastic kernel on ${\bf O}$ given ${\bf A}\times {\bf X}$ defined by 
\[
\hat{O}(do|a,x)=\begin{cases}
\delta_{*}(do) &\mbox{ if $a=n$;}\\
\delta_x(do) &\mbox{ if $a=r$},
\end{cases}
\]
where $\delta_x$ is the Dirac measure at $x$, 
$\hat{O}_0$, defined for each $x\in {\bf X}$ by  $\hat{O}_0(do|x)=\delta_x(do)$, is the observation kernel at the initial time $0$ given by the stochastic kernel on ${\bf O}$ given ${\bf X}$, and $l$ is the one-step cost function in the following form: 
$$l(x,a)=F(x)+G(a)$$ for a $[0,\infty]$-valued measurable function $F$ on ${\bf X} $, and $G(a)\in[0,\infty)$ for each $a\in {\bf A}$.  Cf. \cite{Feinberg:2012}. Typically, $G(n) < G(r)$, but this is not assumed in this paper unless stated otherwise.

In words, at initial time $t=0$, the distribution $\mu$ of the hidden state $X_0$ is given. Given $X_0=x_0$, the initial observation $O_0$ obeys $\hat{O}_0(do|x_0)$. Given $O_0=o_0$, an action $a_0$ is selected based on $(\mu,o_0)$. An action may be one of the following:
\begin{itemize}
\item $n$ (\emph{no intervention}): the hidden state evolves according to the Markov kernel $P$ 
and remains unobserved;
\item $r$ (\emph{restart}): the hidden state is replaced by an independent $Z\sim \nu$ and $Z$ is observed.
\end{itemize}
If at time $t\ge 0$, the hidden state is $X_t=x_t,$ an action $a_t$ is selected based on the observed history $(\mu,o_0,a_0,o_1,a_1,\dots,o_t)$, then $X_{t+1}$ obeys the distribution $\hat{P}(dx'|x_t,a_t)$, and given $X_{t+1}=x_{t+1}$, $O_{t+1}$ obeys the distribution  $\hat{O}(do|a_t,x_{t+1}).$  

When the hidden state is $x$ and an action $a$ is selected at one step, the one-step cost $l(x,a)$
is incurred.

 A policy in this POMDP is given by a sequence $\hat{\pi}=\{\hat{\pi}_t\}_{t=0}^\infty$ of stochastic kernels $\hat{\pi}_t$ on ${\bf A}$ given the observed histories $(\mu,o_0,a_0,o_1,a_1,\dots,o_t).$ For the fixed initial distribution $\mu$ of $X_0$ and a policy $\hat{\pi}$, let $\mathbb{P}_\mu^{\hat{\pi}}$ be the probability (strategic) measure on the canonical sample space of the process $\{(X_t,O_t,A_t)\}_{t=0}^\infty$ in the POMDP model, and $\mathbb{E}_\mu^{\hat{\pi}}$ the corresponding expectation, see \cite[Chapter 4]{O89} or \cite{Feinberg:2012} for the detail.  We are interested in the problem of minimizing the $\gamma$-discounted ($\gamma\in(0,1]$) cost
$\mathbb{E}_\mu^{\hat{\pi}}\Big[\sum_{t=0}^\infty \gamma^t l(X_t,A_t) \Big]$
and the problem of minimizing the long-run average cost
$\limsup_{T\rightarrow \infty}\frac{1}{T}\mathbb{E}_\mu^{\hat{\pi}}\Big[\sum_{t=0}^{T-1} l(X_t,A_t) \Big].$ 
Both problems can be reduced to their counterparts in a (completely observed) MDP, following the general method described in e.g., \cite{Feinberg:2012} or \cite[Chapter 4]{O89}. In general, the state space in the induced belief-state MDP is the space of  (posterior) distributions.  However, the concerned POMDP in this paper can be reduced to an MDP with a simpler state space, as described in the next section.



\section{Reduction to the Belief MDP}

Because the state is observed \emph{precisely and only} at restart times (the initial observation is regarded as a post-restart observation), 
the information available to the controller between restarts
is fully summarized by:
\begin{enumerate}
\item the last observed post-restart state $x\in {\bf X}$;
\item the elapsed time $k\in\mathbb N_0:=\{0,1,2,\dots\}$ since that restart.
\end{enumerate}
Conditioned on $(x,k)$, the belief about the current hidden state is equal to $\delta_x P^k{(dy):=P^k(dy\mid x)}$, {where $P^k$ is the $k$th composition of the Markov kernel $P$ on ${\bf X}$ given ${\bf X}$.} 

Then, define the (fully observed) {MDP model $\{{\bf S},{\bf A},Q,c\}$ with the state space
${\bf S} := {\bf X} \times \mathbb N_0$, the action space ${\bf A}=\{n,r\}$, the transition kernel $Q$ on ${\bf S}$ given ${\bf S}\times {\bf A}$ defined for each  $(x,k)\in {\bf S}$ and $a\in {\bf A}$ by
\begin{align*}Q(dx'\times dk'|(x,k),a)=
\begin{cases}
 \delta_x(dx')\delta_{k+1}(dk')& \mbox{ if $a=n$;}\\
\nu(dx')\delta_0(dk')& \mbox{ if $a=r$,}\end{cases}
\end{align*} 
and the one-step cost function {$c$} defined for each $(x,k)\in {\bf S}$ and $a\in {\bf A}$ by
\begin{equation}
\label{eq:stage_cost_xk}
c(x,k;a)= (P^kF)(x) + G(a),
\end{equation} where $(P^k F)(x):=\int_{\bf X}F(y)P^k(dy\mid x)$ for all $x\in{\bf X}.$}

{One can fully describe $Q$ as follows:}\\
\centerline{$(x,k) \xrightarrow{n} \ (x,k+1)\quad\text{(deterministically),}$}\\
\centerline{$(x,k) \xrightarrow{r}  \ (Z,0),\ \ Z\sim \nu.$}\\
{Consider the MDP model $\{{\bf X},{\bf A},Q,c\}$ described above.
Let $S_t=(Y_t,K_t)$ and $A_t$ be the state and action variable, respectively, of the MDP at time $t$. The bivariate process $\{(S_t,A_t)\}_{t=0}^\infty$ is called the MDP. A policy $\pi=\{\pi_t\}_{t=0}^\infty$ is a sequence of measurable transition kernels $\pi_{t}$ on ${\bf A}$ given ${\bf H}_{t}:=({\bf S}\times {\bf A})^{t}\times {\bf S}$.
A policy $\pi$ is called deterministic stationary if there is a measurable mapping $\varphi$ from ${\bf S}$ to ${\bf A}$ such that $\pi_t(da|s_0,a_0,\dots,s_t)=\delta_{\varphi(s_t)}(da)$ for all $t\ge 0.$ We often identify such a deterministic stationary policy by the underlying measurable mapping $\varphi$. Let $\PP_{x,k}^\pi$ be the strategic measure of $\pi$ with the initial state $(x,k)\in {\bf S}$, and $\EE_{x,k}^\pi$ be the expectation with respect to $\PP_{x,k}^\pi.$ For the detail, see e.g., \cite{O89}.} 

{
We consider the $\gamma$-discounted cost:
\begin{eqnarray*}
V^\pi_\gamma(x,k) & := & \EE_{x,k}^\pi\Big[\sum_{t\ge 0} \gamma^t c(S_t,A_t)\Big]\\
 & = & \EE_{x,k}^\pi\Big[\sum_{t\ge 0} \gamma^t \left\{(P^{K_t}F)(Y_t) + G(A_t)\right\} \Big]
\end{eqnarray*}
for each $\gamma\in(0,1]$, in particular, $\gamma=1$ corresponds to the total undiscounted cost;
and the long-run average cost: 
\begin{align*}
J^\pi(x,k):= \limsup_{N\rightarrow \infty}\frac{1}{N}\EE_{x,k}^\pi\Big[\sum_{t=0}^{N-1}  c(S_t,A_t)\Big].
\end{align*}}

{A policy $\pi^\ast$ is called $\gamma$-discounted (or average) optimal if $V^{\pi^\ast}_\gamma(x,k)= V_\gamma(x,k):= \inf_{\pi} V^{\pi}_\gamma(x,k)$ (resp., $J^{\pi^\ast}(x,k)=J(x,k):=\inf_{\pi} J^\pi(x,k)$) for each $(x,k)\in {\bf S}$.}

{For both the $\gamma$-discounted and average criteria},  the original partially observed control problem is equivalent 
to {its counterpart in the (completely observable) MDP model $\{{\bf S},{\bf A},Q,c\}$.}

In this paper we show under natural conditions that for the $\gamma$-discounted and average cost criteria,  the MDP model $\{{\bf S},{\bf A},Q,c\}$ admits a deterministic stationary optimal policy that has a threshold structure. The precise meaning of the threshold structure is given as follows.
\begin{definition}
\label{Kostia2026Def01} A deterministic stationary policy $\varphi$ has a threshold structure if for each $x\in {\bf X}$, there is some $k^{\star}(x)\in\{0,1,\dots,\infty\}$ such that 
\begin{align}\label{Kostia2026Eqn11}
\varphi(x,k)=\begin{cases}
r & \mbox{ if $k\ge k^{\star}(x)$; }\\
n & \mbox{ if $k<k^{\star}(x)$.} 
\end{cases}
\end{align} 
\end{definition}

\section{Threshold structure for $\gamma$-discounted cost criterion}

In this section, we show the existence of a $\gamma$-discounted optimal deterministic stationary policy that has a threshold structure in the time elapsed since the most recent restart. Furthermore, when the state space is endowed with a partial order, we establish the monotonicity of the threshold in the state. 


\subsection{Existence of optimal policy and threshold structure}
{We consider the MDP model $\{{\bf S},{\bf A},Q,c\}$ with the $\gamma$-discounted cost criterion, where $\gamma\in(0,1]$ is fixed. Recall that 
\[
V_\gamma(x,k):=\inf_{\pi}\ \EE^\pi_{x,k}\Big[\sum_{t\ge 0}\gamma^t\big((P^{K_t}F)(Y_t)+G(A_t)\big)\Big],
\]
for all $(x,k)\in {\bf S}$, defines the value function for the $\gamma$-discounted MDP problem.} 
 
{According to \cite[Theorem 3.1]{FeinbergORL:2021}, which is applicable to the MDP model $\{{\bf S},{\bf A},Q,c\}$ because ${\bf A}$ is finite,} $V_\gamma$ is the minimal $[0,\infty]$-valued measurable solution to the Bellman equation
\begin{equation}
\label{eq:Bellman_xk}
{V_\gamma}(x,k) = (P^kF)(x) +
\end{equation}
\[
\min\Big\{G(n) + \gamma {V_\gamma}(x,k+1), G(r) + \gamma {\int_{{\bf X}}V_\gamma(z,0)\nu(dz)} \Big\}
\]
{for each $(x,k)\in {\bf S}$}.
For notational convenience, define
\begin{equation}
\label{eq:A_def}
{V_{0,\gamma}} := {\int_{{\bf X}}V_\gamma(z,0)\nu(dz)},
\quad
{A(\gamma)}:= G(r)+\gamma V_{0,\gamma}.
\end{equation}
Then, \eqref{eq:Bellman_xk} becomes
\begin{equation}
\label{eq:Bellman_xk_simplified}
V_\gamma(x,k) = (P^kF)(x) + \min\{ G(n)+\gamma V_\gamma(x,k+1), {A(\gamma)} \}.
\end{equation}

\begin{definition}
{For each $\gamma\in(0,1]$, we call the action $n$ (or $r$) $\gamma$-discounted optimal at $(x,k)\in {\bf S}$ if $G(n)+\gamma V_\gamma(x,k+1)\le {A(\gamma)}$ (resp.,  ${A(\gamma)} \le G(n)+\gamma V_\gamma(x,k+1)$.) In other words, an action $a\in {\bf A}$ is $\gamma$-discounted optimal at $(x,k)\in {\bf S}$ if it attains the minimum in the Bellman equation (\ref{eq:Bellman_xk_simplified}) at $(x,k)$.}
\end{definition} 
{
Note that it is possible that the actions $n$ and $r$ are both $\gamma$-discounted optimal at a given $(x,k)$. A deterministic stationary policy $\varphi$ is $\gamma$-discounted optimal if and only if for each $(x,k)\in {\bf S}$, $\varphi(x,k)$ is $\gamma$-discounted optimal at $(x,k)\in {\bf S}$, and such a policy exists, cf. \cite[Theorem 3.1]{FeinbergORL:2021}. }

\smallskip
\begin{assumption}[One-step cost deterioration]
\label{ass:PFgeF}
The function $F$ and the Markov kernel $P$ satisfy
\begin{equation}
\label{eq:PFgeF}
(PF)(x) \ge F(x), \qquad \forall x\in {\bf X},
\end{equation}
where $(PF)(x):=\int_{\bf X} F(y)\,P(dy \mid x)$.
\end{assumption}
\smallskip

\begin{remark}
Assumption~\ref{ass:PFgeF} means that, for the running cost $F$, the one-step prediction is never better than the
current state in expectation.
This condition is very natural in reliability and machine maintenance \cite{BP96}. It is also satisfied in the age of information setting \cite{TM24}.
\end{remark}

Assumption~\ref{ass:PFgeF} implies the monotonicity of the expected state costs
along elapsed time before the next restart.

\medskip
\begin{lemma}[Monotonicity of expected state costs]
\label{lem:PkF_increasing}
Under Assumption~\ref{ass:PFgeF}, for any distribution {$\mu$ on ${\cal B}({\bf X})$ (or simply say on ${\bf X}$)}, 
{ $\int_{\bf X} \int_{\bf X} \mu(dx) P^{k+1}(dy \mid x) F(y)\ge\int_{\bf X} \int_{\bf X} {\mu(dx)} P^{k}(dy \mid x) F(y).$}

\end{lemma}
\medskip
\begin{proof}
We can write
\begin{align*}
&\int_{\bf X} \int_{\bf X} {\mu}(dx) P^{k+1}(dy \mid x) F(y)\\
=&
\int_{\bf X} \int_{\bf X} \int_{\bf X} {\mu}(dx) P^{k}(dz \mid x) P(dy \mid z) F(y)\\
=&
\int_{\bf X} \int_{\bf X} {\mu}(dx) P^{k}(dz \mid x) \int_{\bf X} P(dy \mid z) F(y)\\
\ge& \int_{\bf X} \int_{\bf X} {\mu}(dx) P^{k}(dz \mid x) F(z), 
\end{align*}
where for the last inequality, we used the condition (\ref{eq:PFgeF}).
\end{proof}

\medskip
\begin{theorem}
\label{thm:elapsed_time_threshold}
Suppose {Assumption \ref{ass:PFgeF} holds}, {and let $\gamma\in(0,1]$ be fixed.}
Then there exists a {$\gamma$-discounted} optimal deterministic stationary policy {$\varphi_\gamma^\star$} 
with the following {threshold} structure:
for each $x\in {\bf X}$ there exists a threshold index ${k^\star_\gamma}(x)\in \mathbb N_0\cup\{\infty\}$ 
such that { \begin{align}\label{Kostia2026Eqn08}
\varphi^\star_\gamma(x,k)=\begin{cases}
r & \mbox{ if $k\ge k^{\star}_\gamma(x)$;}\\
n & \mbox{ if $k<k^{\star}_\gamma(x)$,} 
\end{cases}
\end{align}
cf. Definition \ref{Kostia2026Def01}.} 
Equivalently, along each post-restart ray $\{(x,k):k\in\mathbb N_0\}$ the restart region is a tail set.
\end{theorem}
\begin{proof}
Let us first prove that, for each fixed $x{\in {\bf X}}$, the map $k\mapsto {V_\gamma}(x,k)$ is nondecreasing.
Towards this goal, let $\mathcal V$ be the space of {$[0,\infty]$-valued} measurable functions $v: {\bf S}\to\mathbb R$.
 Define {the Bellman operator $T_\gamma$ on ${\cal V}$ by} 
\begin{align}\label{Kostia2026Eqn10}
({T_\gamma} v)(x,k):=(P^kF)(x)+
\end{align}
\[
\min\Big\{G(n)+\gamma v(x,k+1), G(r)+\gamma {\int_{\bf X}v(z,0)\nu(dz)} \Big\}
\]
for each $v\in\mathcal V$ {and $(x,k)\in {\bf S}$}.

Fix $x\in {\bf X}$.  {Observe} that if $v(x,k)$ is nondecreasing in $k\in {\mathbb{N}_0}$ {for some $v\in {\cal V}$}, then so is {$T_\gamma v(x,k)$}.
Indeed, by Lemma~\ref{lem:PkF_increasing}, $k\mapsto (P^kF)(x)$ is nondecreasing.
Also, $k\mapsto v(x,k+1)$ is nondecreasing whenever $k\mapsto v(x,k)$ is.
Therefore, $k\mapsto G(n)+\gamma v(x,k+1)$ is nondecreasing, while the quantity
$G(r)+\gamma{\int_{\bf X}v(z,0)\nu(dz)}$ is independent of $k$.
Hence the {(pointwise)} minimum {between} a nondecreasing function and a constant {function} is nondecreasing, 
and adding $(P^kF)(x)$ preserves {this monotonicity}. 

Now start value iteration from $V^{(0)}\equiv 0$. Then $V^{(0)}(x,k)$ is nondecreasing in $k$ for every $x$.
By the induction principle $V^{(m)}{(x,k)}:={T^m_\gamma} V^{(0)}{(x,k)}$ is nondecreasing in $k$ for every $x$ and every $m\ge0$, {where $T_\gamma^m$ is the $m$th composition of $T_\gamma$}.
{By \cite[Theorem 3.1]{FeinbergORL:2021}}, $V^{(m)}$ increases to {$V_\gamma$ pointwise. Then, $V_\gamma(x,k)$} is also nondecreasing in $k$ for every fixed $x$.

Fix $x\in {\bf X}$. 
{Recall that the action $r$ is $\gamma$-discounted} optimal at $(x,k)$ if and only if
$
{A(\gamma)}\le G(n)+\gamma V_\gamma(x,k+1).
$
{Since $k\mapsto {V_\gamma}(x,k+1)$ is nondecreasing}, the set {$\{k\in\mathbb N_0: A(\gamma)\le G(n)+\gamma V_\gamma(x,k+1)\}$} is a tail set.
Define
\begin{align}\label{eq:kstar_def}
{k^\star_\gamma(x)}:= & \inf\{k\in\mathbb N_0: {A(\gamma)\le G(n)+\gamma V_\gamma(x,k+1)}\}\nonumber\\
= & \inf\{k\in\mathbb N_0:\ \text{$r$ is {$\gamma$-discounted} optimal at } (x,k)\}\nonumber\\ 
\in & \ \mathbb N_0\cup\{\infty\}.
\end{align}
Then {$A(\gamma)> G(n)+\gamma V_\gamma(x,k+1)$} for $k<k^\star_\gamma(x)$ and {$A(\gamma)\le G(n)+\gamma V_\gamma(x,k+1)$} for $k\ge {k^\star_\gamma}(x)$, which is exactly the asserted
threshold structure.

Finally, the policy {$\varphi^\star_\gamma$ defined by (\ref{Kostia2026Eqn08}) is $\gamma$-discounted} optimal {because at each $(x,k)\in {\bf S}$ it selects an $\gamma$-discounted optimal action at $(x,k)$}. 
\end{proof}

\subsection{Monotonicity of threshold in state}

{In this subsection}, let {the Borel state space} ${\bf X}$ be  equipped with a partial order {$\le_{\bf X}$}  such that, 
for every $x\in {\bf X}$, the principal upper set
\[
\uparrow x := \{y\in {\bf X}:\ {x\le_{\bf X} y} \}
\]
is ${\cal B}({\bf X})$-measurable.  {A $[-\infty,\infty]$-valued function $g$ on ${\bf X}$ is nondecreasing  if $g(x)\le g(y)$ whenever $x\le_{\bf X} y$, cf. \cite{Muller2002}.}

The existence of the partial order allows us to investigate
the monotonicity of the threshold index {$k_\gamma^\star$ coming from Theorem \ref{thm:elapsed_time_threshold}, defined in (\ref{eq:kstar_def}).} {Here and throughout this section we consider a fixed $\gamma\in(0,1].$}

\medskip
\begin{theorem}
\label{thm:kstar_monotone_in_x}
{Let ${\bf X}$ be endowed with the partial order $\le_{\bf X}$}, {$F$} be nondecreasing with respect to the partial order {$\le_{\bf X}$}, and the Markov kernel $P$ be \emph{stochastically monotone} 
with respect to {$\le_{\bf X}$}, i.e., for every bounded measurable nondecreasing $g:{\bf X}\to\mathbb R$ the map
$x\mapsto \int_{\bf X} g(y)\,P(dy \mid x)$ is nondecreasing.  Suppose that
 {Assumption \ref{ass:PFgeF} holds}, and {$\gamma\in(0,1]$ is fixed}.  Then {the threshold index $k_\gamma^\star$ in Theorem~\ref{thm:elapsed_time_threshold}, defined by (\ref{eq:kstar_def}),} is nonincreasing with respect to the partial {order $\le_{\bf X}$}, i.e., {for each $x,x'\in {\bf X}$}
\[
x{\le_{\bf X}} x' \quad \Longrightarrow \quad {k^\star_\gamma}(x) \ge {k^\star_\gamma}(x').
\]
\end{theorem}
\medskip
\begin{proof}
The proof is similar to the one of Theorem \ref{thm:elapsed_time_threshold}, and we shall use the notations introduced therein. 

Let us first show that, for each fixed $k$, the map $x\mapsto {V_\gamma}(x,k)$ is nondecreasing
with respect to the partial order {$\le_{\bf X}$.}
{Consider the Bellman operator $T_\gamma$ on ${\cal V}$ defined by (\ref{Kostia2026Eqn10}), where ${\cal V}$ is the set of $[0,\infty]$-valued measurable functions on ${\bf X}$.}
We claim that if, for every $k$ {and $v\in {\cal V}$}, the function $x\mapsto v(x,k)$ is nondecreasing, 
then so is $x\mapsto (T_\gamma v)(x,k)$. {The justification is as follows.}

Fix $k$. By the assumption of stochastic monotonicity applied repeatedly, the kernel $P^k$ is also
stochastically monotone \cite{Daley1968}. Since $F$ is nondecreasing {by assumption}, {for each $N\in \mathbb{N}_0$, $F_N:=\min\{F,N\}$ is bounded nondecreasing, and the map $x\mapsto (P^kF_N)(x)$ is nondecreasing. By the monotone convergence theorem, this implies that}  $x\mapsto (P^kF)(x){=\lim_{N\rightarrow \infty}  (P^kF_N)(x)}$ is nondecreasing. 
Since $x\mapsto v(x,k)$ is nondecreasing {by assumption}, so is $x\mapsto v(x,k+1)$. Hence, $x\mapsto G(n)+\gamma v(x,k+1)$ 
is nondecreasing, while the quantity $G(r)+\gamma\int_{\bf X} v(z,0)\,\nu(dz)$ is constant in $x$. Therefore, 
the minimum of these two terms is nondecreasing in $x$, and adding $(P^kF)(x)$ preserves this property. 
This proves that $(T_\gamma v)(x,k)$ is nondecreasing {in $x$}.

Now start value iteration with ${V^{(0)}}\equiv 0$, which is nondecreasing, and iterate ${V^{(m+1)}}:=T_\gamma{V^{(m)}}$.
Then, by the principle of induction, $x\mapsto {V^{(m)}}(x,k)$ is nondecreasing for all $m$.
{Now, by \cite[Theorem 3.1]{FeinbergORL:2021}, $V_\gamma(\cdot,k)$ is} 
the pointwise limit of the sequence {$\{V^{(m)}(\cdot,k)\}_{m=0}^\infty$} of nondecreasing functions, and is therefore, nondecreasing.

{Recall that the} action $r$ is {$\gamma$-discounted} optimal at $(x,k)$ if and only if
\begin{align*}
{A(\gamma)} \le G(n)+\gamma {V_\gamma}(x,k+1).
\end{align*}
The left-hand side is constant in $x$, and the right-hand side is nondecreasing in $x$ as proven above. 
Hence, for each fixed $k$, the set
\[
R_k:=\{x\in {\bf X}:\ \text{$r$ is {$\gamma$-discounted optimal} at }(x,k)\}
\]
is an upper set: if $x\in R_k$ and {$x\le_{\bf X} x'$}, then $x'\in R_k$.

Let us now show the monotonicity of $k^\star_\gamma$. Recall the definition
of $k^\star_\gamma(x)$ in (\ref{eq:kstar_def}).
Fix $x\le_{\bf X}x'$ and $k\in \mathbb{N}_0$. If $r$ is $\gamma$-discounted optimal at $(x,k)$, 
i.e., $x\in R_k$, then, since $R_k$ is an upper set, $x'\in R_k$ and thus $r$ is $\gamma$-discounted optimal at $(x',k)$.
Therefore, 
$\{k:\ r\mbox{ is $\gamma$-discounted optimal at }(x,k)\}\subseteq \{k:\ r\mbox{ is $\gamma$-discounted optimal at }(x',k)\}$. Since both sets are tail sets by Theorem~\ref{thm:elapsed_time_threshold}, $k^\star_\gamma(x')\le k^\star_\gamma(x)$.
\end{proof}
\smallskip

Note that Assumption~\ref{ass:PFgeF} can be guaranteed by the assumption of `no-downward-move' or
`irreversible deterioration' in the reliability terminology \cite{BP96}; see the next lemma.

\medskip
\begin{lemma}
\label{lem:ND_implies_PF_ge_F}
Let {${\bf X}$} be equipped with the partial order $\le_{\bf X}$, and $F$ be nondecreasing with respect to $\le_{{\bf X}}$.
Let the Markov kernel $P$ satisfy the \emph{no-downward-move}
condition
\begin{equation}
\label{eq:ND}
P(\uparrow x\mid x)=1 \qquad \forall x\in {\bf X}.
\end{equation}
Then  Assumption \ref{ass:PFgeF} holds.
\end{lemma}
\medskip
\begin{proof}
Fix $x\in {\bf X}$. By \eqref{eq:ND}, the measure $P(\cdot\mid x)$ is supported on $\uparrow x$, i.e.,
$P({\bf X}\setminus \uparrow x \mid x)=0$. Since $F$ is nondecreasing, we have
\[
y\; \in \; \uparrow x \ \Longrightarrow\ {x\le_{\bf X}y} \ \Longrightarrow\ F(y)\ge F(x).
\]
Therefore, $F(y)\ge F(x)$ holds for $P(\cdot\mid x)$-almost {all $y\in{\bf X}$}, and integrating {with respect to $P(\cdot\mid x)$} yields
\[
(PF)(x)=\int_{\bf X} F(y)P(dy| x)\ge\int_{\bf X} F(x)P(dy| x)=F(x),
\]
which establishes (\ref{eq:PFgeF}).
\end{proof}
\medskip

\section{Threshold structure for average cost criterion}

Under Assumption  \ref{ass:PFgeF}, let us consider $k_\gamma^\star$ coming from Theorem \ref{thm:elapsed_time_threshold}, defined in (\ref{eq:kstar_def}). In this section, we first show that for some $\gamma_0\in(0,1)$ and $\bar{K}\ge 1$, $k_\gamma^\star(x)<\bar{K}$ for all $\gamma\in [\gamma_0,1)$ and $x\in {\bf X}$, under some additional assumptions. This result allows us to bound the relative value function, and then to deduce the existence of an average optimal deterministic stationary policy with the threshold structure (cf. Definition \ref{Kostia2026Def01}) later in this section.   

\subsection{Uniform {boundedness} of the optimal threshold}  

Let us make the following additional assumptions:


\begin{assumption}[Geometric ergodicity of state cost]
\label{ass:UniGeo}
There exist constants $L\in\mathbb R$, $C<\infty$, and $\rho\in(0,1)$ such that
\begin{equation}
\label{eq:A2_geo}
\sup_{x\in {{\bf X}}}\big|(P^tF)(x)-L\big|\le C\rho^t,\qquad \forall t\in\mathbb N_0.
\end{equation}
\end{assumption}
Assumption \ref{ass:UniGeo} is satisfied if e.g., ${\bf X}$ is finite, $F$ is finite-valued, and $P$ is unichain and aperiodic.  
{Assumption \ref{ass:UniGeo} implies in particular that the $[0,\infty]$-valued function $F$ is bounded (by e.g., $L+C$). Together with Assumption \ref{ass:PFgeF}, Assumption \ref{ass:UniGeo} further implies for each $x\in {\bf X}$ that the sequence $\{(P^kF)(x)\}_{k=0}^\infty$ increases to $L$, and in particular,
\begin{equation}\label{eq:PkF_le_L}
(P^kF)(x)\le L,\qquad \forall x\in X,\ k\in\mathbb N_0,
\end{equation}
and consequently $\nu P^kF\le L$ for all $k$.}
 
Let $\varphi^n$ be the deterministic stationary policy defined by $\varphi^n(x,k)\equiv n$. This is the policy that always selects the action $n$, and is referred to as the always-$n$ policy.   The next lemma asserts that if for some $\gamma\in(0,1)$, $k_\gamma^\star(x)=\infty$ for some $x\in {\bf X}$, then $\varphi^n$ is $\gamma$-discounted optimal. 
\medskip
\begin{lemma}\label{Kostia2026Lem01}
Suppose Assumptions \ref{ass:PFgeF} and \ref{ass:UniGeo} are satisfied, and let $\gamma\in (0,1)$ be given. Consider $k_\gamma^\star$ coming from Theorem \ref{thm:elapsed_time_threshold}, defined in (\ref{eq:kstar_def}).  If $k^\star_\gamma(x)=\infty$ for some $x,$ then the always-$n$ policy $\varphi^n$ is $\gamma$-discounted optimal, i.e., $V_\gamma^{\varphi^n}(x,k)=V_\gamma(x,k)$ for all $x\in {\bf X}$ and $k\ge 0.$
\end{lemma}
\medskip
\begin{proof}
Let $x\in {\bf X}$ be given so that $k^\star_\gamma(x)=\infty$. Then 
\begin{align}\label{Kostia2026Eqn01}
V_\gamma(x,k)=V_\gamma^{\varphi^n}(x,k)= \sum_{t=0}^{\infty}\gamma^{t} (P^{k+t}F)(x) + \frac{G(n)}{1-\gamma}
\end{align}
for all $k\ge 0$,
and
\begin{equation}
\label{Kostia2026Eqn02}
G(n)+\gamma V_\gamma(x,k+1)=G(n)+
\end{equation}
\[
\gamma \left\{ \sum_{t=0}^\infty  \gamma^t(P^{k+1+t}F)(x)+\frac{G(n)}{1-\gamma}\right\} < G(r)+\gamma V_{0,\gamma},
\]
for all $k\ge 0$, where the equality holds by (\ref{Kostia2026Eqn01}) and the inequality is by (\ref{eq:kstar_def}) and the assumption that $k^\star_\gamma(x)=\infty$.

Since under Assumptions \ref{ass:PFgeF} and \ref{ass:UniGeo}    $(P^k F)(x)$ increases to $L$ as $k\uparrow\infty$,
\begin{align}\label{Kostia2026Eqn05}
 &\lim_{k\rightarrow \infty}\left\{G(n)+\gamma \left\{ \sum_{t=0}^\infty  \gamma^t(P^{k+1+t}F)(x)+\frac{G(n)}{1-\gamma}\right\}\right\} \nonumber\\
=& G(n)+\gamma \left\{ \sum_{t=0}^\infty  \gamma^tL+\frac{G(n)}{1-\gamma}\right\}
 \le G(r)+\gamma V_{0,\gamma}.
\end{align}
where the last inequality holds by (\ref{Kostia2026Eqn02}). 

Now let $y\in {\bf X}$ be arbitrarily fixed. 
For each $k\ge 0$, 
\begin{align*}
&G(n)+\gamma V_\gamma(y,k+1)\le G(n) +\gamma V_\gamma^{\varphi^n}(y,k+1)\\
=&  G(n)+\gamma \left\{ \sum_{t=0}^\infty  \gamma^t(P^{k+1+t}F)(y)+\frac{G(n)}{1-\gamma}\right\}\\
\le & G(n)+\gamma \left\{ \sum_{t=0}^\infty  \gamma^tL+\frac{G(n)}{1-\gamma}\right\}
\le G(r)+\gamma V_{0,\gamma}. 
\end{align*}
where the first inequality follows from the definition of the (optimal) value function $V_\gamma$, the second inequality hold because $P^kF(y)$ increases to $L$ as $k\uparrow \infty$ (under Assumptions \ref{ass:PFgeF} and \ref{ass:UniGeo}), and the last inequality holds by (\ref{Kostia2026Eqn05}). Thus, the action $n$ is $\gamma$-discounted optimal at $(y,k)$ for any $y\in{\bf X}$ and $k\ge 0.$ 
\end{proof}

\begin{assumption}[Transient gain dominates]
\label{ass:TG}
There exist $m\in\mathbb N_0$ and $\eta>0$ such that
\begin{equation}\label{eq:A3_TG}
\sum_{j=0}^{m}\Big(L-\nu P^jF\Big)\ \ge\ \Delta G+\eta.
\end{equation}
where $\Delta G := G(r)-G(n) > 0$.
\end{assumption}
\smallskip

{It will be seen in the proof of Theorem \ref{thm:uni_bound_threshold} that Assumption \ref{ass:TG} excludes the possibility that $k_\gamma^\ast(x)=\infty$ for any $x\in {\bf X}$ and $\gamma\in (0,1)$ sufficiently close to $1$.} 

Under the above assumptions, we can produce a uniform bound on the threshold {$k_\gamma^\star(x)$ with respect to all $x\in {\bf X}$ and all $\gamma\in (0,1)$ sufficiently close to $1$.}
\medskip
\begin{theorem}
\label{thm:uni_bound_threshold}
Let {Assumptions \ref{ass:PFgeF},~\ref{ass:UniGeo} and \ref{ass:TG} hold, and consider $k_\gamma^\star$ coming from Theorem \ref{thm:elapsed_time_threshold}, defined in (\ref{eq:kstar_def}).}
Then there exist $\gamma_0\in(0,1)$ and a {nonnegative} integer $K<\infty$ such that  
\begin{equation}\label{eq:uniform_bound_conclusion}
\sup_{x\in {\bf X}} {k^\star_\gamma}(x)< {K+1 =: \bar{K}}, \quad {\forall \gamma\in[\gamma_0,1)}.
\end{equation}
In particular, the optimal threshold {$k_\gamma^\star(x)$} is uniformly bounded in $x{\in{\bf X}}$ and 
  uniformly bounded in $\gamma{\in [\gamma_0,1)}$.
\medskip
Moreover, one may take
\begin{equation}\label{eq:delta0_def}
K=\min\Big\{t\in\mathbb N_0: C\rho^{t}\le \delta_0/4\Big\},
\ \mbox{with} \
\delta_0:=\frac{\eta\,\gamma_0^{m}}{m+1},
\end{equation}
where $\gamma_0$ {can be any constant in $(0,1)$ that} satisfies the {inequalities} 
$(1-\gamma_0)\Delta G\le \eta\gamma_0^{m+1}/(2(m+1))$ {and 
$(1-\gamma_0^{m+1})\Delta G <\eta \gamma_0^{m+1}.$ (Such a constant exists because the left-hand sides (or right-hand sides) of the previous two inequalities) converge to $0$ (resp., a positive constant) as $\gamma_0$ increases to $1$.)} 
\end{theorem}
\medskip
\begin{proof}
Since $F$ is bounded under Assumption \ref{ass:UniGeo}, for each $\gamma\in(0,1)$, $V_\gamma$ is bounded. Until the end of this proof, we consider $\gamma\in(0,1)$ unless stated otherwise.
Let $a_\gamma:=(1-\gamma){A(\gamma)}$, where ${A(\gamma)}$ is defined in (\ref{eq:A_def}).

Consider the {deterministic} stationary policy ${\varphi}^{(m)}$ {defined for all $x\in {\bf X}$ and $k\in\mathbb{N}_0$ by $$\varphi^{(m)}(x,k):=\begin{cases}
n &\mbox{ if $0\le k\le m-1$},\\
r &\mbox{ if $k\ge m$,}
\end{cases}$$
where $m$ is as in Assumption \ref{ass:TG}.}
{Note that starting with any $(x,0)\in {\bf S}$, the policy $\varphi^{(m)}$} chooses $n$ at {$S_k=(Y_k,K_k)=(x,k)$} for $k=0,1,\dots,m-1$
and chooses $r$ at {$S_m=(x,m)$.  After $r$ is applied, the post-restart state variable $Y_{m+1}$ obeys the distribution $\nu$, the elapsed time $K_{m+1}$ is reset to $0$, and this pattern is repeated.} 
Now starting from a restart state distributed as $\nu$ at elapsed time $0$, 
one cycle lasts $m+1$ steps and incurs the expected discounted cost
\[
B_m(\gamma):=\sum_{j=0}^{m}\gamma^j\,\nu P^jF\;+\;\sum_{j=0}^{m-1}\gamma^j G(n)\;+\;\gamma^m G(r).
\]
Since the post-restart state is again distributed as $\nu$, the total expected {$\gamma$-discounted} cost from restart 
under ${\varphi}^{(m)}$ is equal to
\begin{equation}\label{eq:Jm_def}
J_m(\gamma){:= \int_{\bf X}  V_\gamma^{\varphi^{(m)}}(x,0)\nu(dx)} =\frac{B_m(\gamma)}{1-\gamma^{m+1}}.
\end{equation}
 {Note that  $V_{0,\gamma}=\int_{\bf X}  V_\gamma (x,0)\nu(dx) \le J_m(\gamma)$.}

Define weights
\[
w_j(\gamma):=\frac{(1-\gamma)\gamma^j}{1-\gamma^{m+1}},\qquad j=0,1,\dots,m,
\]
so that $\sum_{j=0}^{m} w_j(\gamma)=1$ and $w_j(\gamma)\ge w_m(\gamma)$ for all $j\le m$.
Then, multiplying \eqref{eq:Jm_def} by $(1-\gamma)$ and rearranging {yield}
\begin{equation}\label{eq:(1-g)Jm_expand}
(1-\gamma)J_m(\gamma)=\sum_{j=0}^m w_j(\gamma)\,\nu P^jF\;+\;G(n)\;+\;w_m(\gamma)\Delta G.
\end{equation}
{Define $d_t := L - \nu P^t F$ for all $t\ge 0$.}
Using \eqref{eq:PkF_le_L}, we have $d_j = L-\nu P^jF\ge 0$ for all $j\le m$, and hence
\[
\sum_{j=0}^m w_j(\gamma)\,\nu P^jF
= L-\sum_{j=0}^m w_j(\gamma)d_j
\le L-w_m(\gamma)\sum_{j=0}^m d_j.
\]
Moreover,
\[
w_m(\gamma)=\frac{(1-\gamma)\gamma^m}{1-\gamma^{m+1}}
=\frac{\gamma^m}{1+\gamma+\cdots+\gamma^m}\ \ge\ \frac{\gamma^m}{m+1}.
\]
Combining the last two inequalities with \eqref{eq:(1-g)Jm_expand} and \eqref{eq:A3_TG} yields
\[
(1-\gamma)J_m(\gamma)\ \le\ L+G(n)-\frac{\eta\,\gamma^m}{m+1}.
\]
Therefore, for any fixed $\gamma_0\in(0,1)$ and all $\gamma\in[\gamma_0,1)$,
\begin{equation}\label{eq:(1-g)V0_bound}
(1-\gamma)V_{0,\gamma}\ \le\ (1-\gamma)J_m(\gamma)\ \le\ L+G(n)-\delta_0,
\end{equation}
where $\delta_0 := \eta\,\gamma_0^{m}/(m+1)$. Using \eqref{eq:A_def} and \eqref{eq:(1-g)V0_bound},
\begin{align}
a_\gamma
&=(1-\gamma){A(\gamma)}=(1-\gamma)G(r)+\gamma(1-\gamma)V_{0,\gamma}\nonumber\\
&\le (1-\gamma)\big(G(n)+\Delta G\big)+\gamma\big(L+G(n)-\delta_0\big)\nonumber\\
&= G(n)+\gamma L-\gamma\delta_0 + (1-\gamma)\Delta G.\nonumber
\end{align}
Choose $\gamma_0\in(0,1)$ close enough to $1$ so that
\begin{align}\label{Kostia2026Eqn06}
(1-\gamma_0)\Delta G\ \le\ \frac{\gamma_0\,\delta_0}{2}
\end{align}
$\Big(\text{equivalently }\
(1-\gamma_0)\Delta G\le \frac{\eta\,\gamma_0^{m+1}}{2(m+1)}\Big)$
and 
\begin{align}\label{Kostia2026Eqn07}
 (1-\gamma_0^{m+1})\Delta G<\eta \gamma_0^{m+1}.
\end{align}
(The above two inequalities are the same as the two at the end of the statement of this theorem.)

Then {from (\ref{Kostia2026Eqn06}),} for all $\gamma\in[\gamma_0,1)$ we have $(1-\gamma)\Delta G\le \gamma\delta_0/2$, and
\begin{equation}\label{eq:a_gamma_bound}
a_\gamma\ \le\ G(n)+\gamma\Big(L-\frac{\delta_0}{2}\Big),
\quad \forall \gamma\in[\gamma_0,1).
\end{equation}

\medskip
Let $K$ be as in \eqref{eq:delta0_def}, so that by \eqref{eq:A2_geo},
\begin{equation}\label{eq:PK_lower}
(P^K F)(x)\ \ge\ L-C\rho^K\ \ge\ L-\frac{\delta_0}{4},
\quad \forall x\in {{\bf X}}.
\end{equation}
Combining \eqref{eq:a_gamma_bound}--\eqref{eq:PK_lower} yields, for all $\gamma\in[\gamma_0,1)$ and all $x\in {{\bf X}}$,
\[
a_\gamma\ \le\ G(n)+\gamma\Big(L-\frac{\delta_0}{2}\Big)
\ <\ G(n)+\gamma\Big(L-\frac{\delta_0}{4}\Big)\
\]
\begin{equation}\label{eq:a_vs_PK}
\le\ G(n)+\gamma(P^K F)(x).
\end{equation}

{Up to the end of this proof,} fix $\gamma\in[{\gamma_0},1)$ and $x\in {{\bf X}}$. If $k^\star_\gamma(x)=0$, then {$k^\star_\gamma(x)<\bar{K}=K+1$} follows {automatically}.

Next, suppose $1 \le k^\star_\gamma(x) < \infty$.
By the definition of $k^\star_\gamma(x)$ (see (\ref{eq:kstar_def})), {the} action $r$ is {$\gamma$-discounted} optimal at $(x,k^\star_\gamma(x))$ and {the} action $n$ is {$\gamma$-discounted} optimal at $(x,k^\star_\gamma(x)-1)$.
Using \eqref{eq:Bellman_xk} at these two states gives
\begin{align}
&{A(\gamma)} \le G(n)+\gamma V_\gamma\big(x,k^\star_\gamma(x)+1\big), \label{eq:sandwich1}\\
&G(n)+\gamma V_\gamma\big(x,k^\star_\gamma(x)\big) \le {A(\gamma)}. \label{eq:sandwich2}
\end{align}
Since {by Theorem \ref{thm:elapsed_time_threshold}}, $r$ is {$\gamma$-discounted} optimal at all $(x,k)$ with $k\ge k^\star_\gamma(x)$, the Bellman equation {(\ref{eq:Bellman_xk_simplified})} yields
\[
V_\gamma(x,k)=(P^kF)(x)+{A(\gamma)},\qquad k\ge k^\star_\gamma(x).
\]
Substituting this identity into \eqref{eq:sandwich1}--\eqref{eq:sandwich2} gives
\begin{equation}\label{eq:sandwich_final}
G(n)+\gamma(P^{k^\star_\gamma(x)}F)(x)\ \le\ a_\gamma\ \le\ G(n)+\gamma(P^{k^\star_\gamma(x)+1}F)(x).
\end{equation}

{Suppose for contradiction that} $k^\star_\gamma(x)> K$ {with $K$ being given by (\ref{eq:delta0_def})}. Then by the monotonicity (Assumption~\ref{ass:PFgeF}) we have
$(P^{k^\star_\gamma(x)}F)(x)\ge (P^K F)(x)$, and hence, {by the left inequality in \eqref{eq:sandwich_final},}
$G(n)+\gamma(P^K F)(x)\le a_\gamma$, contradicting \eqref{eq:a_vs_PK}.
Therefore, $k^\star_\gamma(x) \le K{<\bar{K}}$, as desired.

{Finally,} let us show that $k^\star_\gamma(x) < \infty$.   
Suppose {for contradiction that}  $k^\star_\gamma(x) = \infty$. {By Lemma \ref{Kostia2026Lem01},} 
the action $n$ is {$\gamma$-discounted} optimal at $({y},k)$ for all {$k\ge 0$ and $y\in{\bf X}$, and $V_\gamma^{\varphi^n}=V_\gamma$, where $\varphi^n$ is the always-$n$ policy.}
Substituting the explicit expression of {$V^{\varphi^n}_\gamma$}, i.e.,
\[
V_\gamma^{{\varphi^n}}({y},k) \; = \; \sum_{t=0}^{\infty}\gamma^{t} (P^{k+t}F)({y}) + \frac{G(n)}{1-\gamma}
\]
into the Bellman equation
{\begin{align*}
V_\gamma^{\varphi^n}(y,k)=&(P^kF)(y)+\min\Big\{G(n)+\gamma V^{\varphi^n}(y,k+1), \\
&G(r)+\gamma \int_{{\bf X}} V^{\varphi^n}(z,0)\nu(dz)  \Big\}
\end{align*}}
and canceling the common terms 
$\gamma G(n)/(1-\gamma)$ implies:
\begin{equation}\label{eq:always_n_condition}
\gamma \sum_{t=0}^{\infty}\gamma^{t}\Big((P^{k+1+t}F)({y})-\nu P^{t}F\Big) \le \Delta G, \
\forall k \in {\mathbb N}_0, y\in {\bf X},
\end{equation}
where $\nu P^{t}F = \int_X (P^{t}F)(z)\,\nu(dz)$. By the dominated convergence theorem,
when $k \to \infty$, (\ref{eq:always_n_condition}) becomes
$$
\gamma \sum_{t=0}^{\infty}\gamma^{t}\Big(\,L\;-\;\nu P^t F\,\Big) \le \Delta G.
$$
{Recall that $d_t = L - \nu P^t F\ge 0$ for all $t\ge 0$}. Then, by  Assumption~\ref{ass:TG},
we can write a string of inequalities
$$
\Delta G \ge \gamma \sum_{t=0}^{\infty}\gamma^{t} d_t 
\ge \gamma^{m+1} \sum_{t=0}^{m} d_t \ge \gamma^{m+1} (\Delta G + \eta),
$$ 
{and hence  $\Delta G (1-\gamma^{m+1})\ge \gamma^{m+1}\eta.$  
Since $\gamma\in [\gamma_0,1)$, this leads to a contradiction  against (\ref{Kostia2026Eqn07}).}
Thus, 
$k^\star_\gamma(x) < \infty$. 

{Since $x\in {\bf X}$ and $\gamma\in[\gamma_0,1)$ were arbitrarily fixed,} from the above considerations 
$$
\sup_{x\in {\bf X},\gamma\in[\gamma_0,1)} k^\star_\gamma(x) \le K < {K+1=\bar{K}}
$$ 
and \eqref{eq:uniform_bound_conclusion} holds.
\end{proof}
\medskip
\begin{remark}
If $\Delta G>0$, a uniform bound over all $\gamma\in(0,1)$ generally cannot hold:
for sufficiently small $\gamma$ the future is heavily discounted and paying $\Delta G$
can make restarting suboptimal everywhere (so in that case $k^\star_\gamma(x)=\infty$).
\end{remark}
\medskip


\subsection{Uniform {boundedness} of the relative value {function} and optimality results}

{In this subsection, we show under Assumptions \ref{ass:PFgeF},~\ref{ass:UniGeo} and \ref{ass:TG} that there is an average optimal deterministic stationary policy with the threshold structure. We will use the vanishing discount approach established in \cite{Schal1993}. To this end,}  let us establish the {boundedness} of the relative value function. {Throughout this subsection, unless stated otherwise, $\gamma\in(0,1)$.}
\medskip
\begin{theorem}
\label{thm:bound_relative_value}
Let {Assumption \ref{ass:PFgeF} hold.}
Assume further that {the $[0,\infty]$-valued function $F$ is bounded,} and there exist $\gamma_0\in(0,1)$ and a {positive} integer ${\bar{K}}<\infty$ such that
\begin{equation}
\label{eq:unifK}
\sup_{x\in {\bf X}} k^\star_\gamma(x) <{\bar{K}},
\qquad \forall \gamma\in [\gamma_0,1),
\end{equation}
where $k^\star_\gamma(\cdot)$ is the optimal restart threshold defined in (\ref{eq:kstar_def}). 
(Under {Assumptions \ref{ass:PFgeF}, \ref{ass:UniGeo} and \ref{ass:TG}}, {$F$ is boundeded, and moreover,} such $\gamma_0,{\bar{K}}$ exist by 
Theorem~\ref{thm:uni_bound_threshold}.) Define 
\[
m_\gamma \;:=\; \inf_{(x,k)\in {\bf X}\times\mathbb N_0} V_\gamma(x,k),
\]
{where $V_\gamma$ is the value function for the $\gamma$-discounted criterion}.
Then
\[
\sup_{\gamma\in(0,1)}\ \sup_{(x,k)\in {\bf X}\times \mathbb N_0}\ \bigl(V_\gamma(x,k)-m_\gamma\bigr)
\;<\;\infty.
\]
Specifically, with
\[
\bar c \;:=\; \|F\|_\infty + \max\bigl\{|G(n)|,\ |G(r)|\bigr\},
\]
{where for any bounded function, $\|\cdot\|_\infty$ denotes its supremum norm,}
one has the explicit uniform bound
\begin{equation}
\label{eq:explicit_bound_inf}
\sup_{\gamma\in(0,1)}\ \sup_{(x,k)\in {\bf X}\times \mathbb N_0}\ \bigl(V_\gamma(x,k)-m_\gamma\bigr)
\le
\max\left\{\frac{2\bar c}{1-\gamma_0},\ 4 {\bar{K}} \bar c\right\}.
\end{equation}
\end{theorem}
\medskip
\begin{proof}
Fix $\gamma\in(0,1)$. Since $F$ is bounded, we have
$|(P^kF)(x)|\le \|F\|_\infty$ for all $x,k$, and  
${0\le}(P^kF)(x)+G(a){\le \bar c}$ {for all $x,k$ and $a$}.
Hence, for any policy $\pi$ and any initial state $(x,k)$,
\[
{V_\gamma^\pi(x,k)=} \EE^\pi_{x,k}\!\left[\sum_{t\ge 0}\gamma^t\bigl((P^{K_t}F)({Y}_t)+G({A}_t)\bigr)\right]
\]
\[
\le \sum_{t\ge 0}\gamma^t\,\bar c \;=\; \frac{\bar c}{1-\gamma},
\]
and therefore,
\begin{equation}
\label{eq:crudeV}
\|V_\gamma\|_\infty \le \frac{\bar c}{1-\gamma},
\
\spn(V_\gamma):=\sup_{s{\in {\bf S}}} V_\gamma(s)-\inf_{s{\in {\bf S}}} V_\gamma(s)\le \frac{2\bar c}{1-\gamma}.
\end{equation}

\smallskip
\noindent\emph{Case 1: $\gamma\in(0,\gamma_0]$.}
By \eqref{eq:crudeV}, $\spn(V_\gamma)\le 2\bar c/(1-\gamma)\le 2\bar c/(1-\gamma_0)$.
Since $V_\gamma(s)-m_\gamma\le \spn(V_\gamma)$ for all $s{=(x,k)\in {\bf S}}$, we obtain
\[
\sup_{s{\in {\bf S}}}\bigl(V_\gamma(s)-m_\gamma\bigr)\le \frac{2\bar c}{1-\gamma_0},
\qquad \gamma\in(0,\gamma_0].
\]

\smallskip
\noindent\emph{Case 2: $\gamma\in[\gamma_0,1)$.}
Let $V_{0,\gamma}$ and $A{(\gamma)}=G(r)+\gamma V_{0,\gamma}$ be as in (\ref{eq:A_def}). 
Fix $(x,k)\in {{\bf S}=}{\bf X}\times \mathbb N_0$ and define
\[
\tau := \bigl(k^\star_\gamma(x)-k\bigr)_+{:=\max\{0,k^\star_\gamma(x)-k\}} \in\{0,1,\dots,{\bar{K}}-1\}.
\]
By repeated use of the Bellman equation along the deterministic $n$-transitions,
and using that {the action} $r$ is {$\gamma$-discounted} optimal at $(x,k^\star_\gamma(x))$ while {the action} $n$ is {$\gamma$-discounted} optimal for all
$(x,j)$ with $j<k^\star_\gamma(x)$ (this is exactly the meaning of the threshold 
index (\ref{eq:kstar_def})), we obtain the identity
\begin{eqnarray*}
V_\gamma(x,k) & = &
\sum_{t=0}^{\tau-1}\gamma^t\Bigl((P^{k+t}F)(x)+G(n)\Bigr) \\
& & +\gamma^\tau\Bigl((P^{k+\tau}F)(x)+{A(\gamma)} \Bigr).
\end{eqnarray*}
Substituting ${A(\gamma)}=G(r)+\gamma V_{0,\gamma}$ into the above equation and subtracting $V_{0,\gamma}$
from both sides yields
\[
V_\gamma(x,k)-V_{0,\gamma} =
\sum_{t=0}^{\tau-1}\gamma^t\Bigl((P^{k+t}F)(x)+G(n)\Bigr)
\]
\[
+\gamma^\tau\Bigl((P^{k+\tau}F)(x)+G(r)\Bigr)
+\bigl(\gamma^{\tau+1}-1\bigr)V_{0,\gamma}.
\]
Taking absolute values and using ${0\le (P^{l}F)(x)+G(a)}\le \bar c$ {for all $x,a$ and $l\in\mathbb{N}_0$} gives
\[
|V_{{\gamma}}(x,k)-V_{0,\gamma}|
\le \frac{1-\gamma^{\tau+1}}{1-\gamma} \bar c + (1-\gamma^{\tau+1})|V_{0,\gamma}|.
\]
By \eqref{eq:crudeV}, $|V_{0,\gamma}|\le \|{V_\gamma}\|_\infty\le \bar c/(1-\gamma)$. Since
$1-\gamma^{\tau+1}\le (\tau+1)(1-\gamma)$, we have
\[
|{V_\gamma}(x,k)-V_{0,\gamma}|
\le (\tau+1)\bar c + (\tau+1)\bar c
= 2(\tau+1)\bar c
\le 2 {\bar{K}} \bar c.
\]
Taking the supremum over $s=(x,k){\in {\bf S}}$ yields
\begin{equation}
\label{eq:unif_to_V0}
\sup_{s{\in {\bf S}}}|{V_\gamma}(s)-V_{0,\gamma}| \le 2 {\bar{K}} \bar c,
\qquad \gamma\in[\gamma_0,1).
\end{equation}
Consequently,
\[
\spn({V_\gamma})
=\sup_{s\in {\bf S}} {V_\gamma}(s)-\inf_{s\in {\bf S}} {V_\gamma}(s) \le 
\]
\[
\bigl(V_{0,\gamma} + 2{\bar{K}}\bar c\bigr)-\bigl(V_{0,\gamma} - 2{\bar{K}}\bar c\bigr)
=4{\bar{K}}\bar c.
\]
Since ${V_\gamma}(s)-m_\gamma\le \spn({V_\gamma})$ for all $s\in {\bf S}$, we obtain
\[
\sup_{s\in {\bf S}}\bigl({V_\gamma}(s)-m_\gamma\bigr)\le 4 {\bar{K}} \bar c,
\qquad \gamma\in[\gamma_0,1).
\]
Combining the bounds from Cases 1 and 2 proves \eqref{eq:explicit_bound_inf}.
\end{proof}
\medskip
Next, we can extend the established {threshold} structure of the optimal policy from
the {$\gamma$-discounted cost} criterion to the long-run average {cost} criterion.
\medskip
\begin{theorem}
\label{thm:average_opt}
Let Assumptions \ref{ass:PFgeF},~\ref{ass:UniGeo} and \ref{ass:TG} hold. Then the MDP model $\{{\bf S},{\bf A},Q,c\}$ admits an average optimal deterministic stationary policy $\varphi$ that has the threshold structure, i.e., for each $x\in {\bf X}$, there is some $k^{\star}(x)\in\{0,1,\dots,\bar{K}-1\}$ such that 
\begin{align*}
\varphi(x,k)=\begin{cases}
r & \mbox{ if $k\ge k^{\star}(x)$;}\\
n & \mbox{ if $k<k^{\star}(x)$}.
\end{cases}
\end{align*}


\end{theorem} 
\medskip
\begin{proof}
  We consider for each $\gamma\in (0,1)$ the $\gamma$-discounted optimal deterministic stationary policy $\varphi_\gamma^\star$ and the threshold index $k_\gamma^\star$ as in Theorem \ref{thm:elapsed_time_threshold}. 
By Theorem \ref{thm:bound_relative_value}, \cite[Assumption (B)]{Schal1993} is satisfied. Moreover, \cite[General Assumption]{Schal1993} holds because $F$ is bounded under  Assumption \ref{ass:UniGeo}. Finally, \cite[Assumption (S)]{Schal1993} holds because ${\bf A}=\{n,r\}$ is finite. Consequently,  \cite[Theorem 3.8]{Schal1993} is applicable. 

 Fix $x\in {\bf X}$. According to  \cite[Theorem 3.8]{Schal1993}, with a diagonal argument, there is an average optimal deterministic stationary policy $\varphi$ such that $\varphi(x,k)=\lim_{l\rightarrow\infty}\varphi_{\gamma_l}^\star(x,k)$ for all $k\in\mathbb{N}_0$ for a sequence $\{\gamma_l\}_{l=1}^\infty\subseteq(0,1)$ converging to $1$. 
By Theorem \ref{thm:uni_bound_threshold}, for some positive integer $\bar{K}$, $k_{\gamma_l}^\star(x)\in \{0,\dots,\bar{K}-1\}$ for all large enough $l$. By taking a subsequence if necessary, one can assume without loss of generality that $\{k_{\gamma_l}^\star(x)\}_{l=1}^\infty$ converges to some $k^\star(x)\in \{0,\dots,\bar{K}-1\}$, and moreover, for some positive integer $l_0$, 
$
k_{\gamma_l}^\star(x)=k_{\gamma_{l_0}}^\star(x)=k^\star(x)$ for all $l\ge l_0$.

Now for all $l\ge l_0$, by (\ref{Kostia2026Eqn08}), we see for each $k\in\mathbb{N}_0$ that $$\varphi_{\gamma_l}^{\star}(x,k)=\begin{cases}
r &\mbox{ if $k\ge k_{\gamma_l}^\star(x)=k^\star_{\gamma_{l_0}}(x)=k^\star(x)$;}\\
n &\mbox{ if $k< k_{\gamma_l}^\star(x)=k_{\gamma_{l_0}}^\star(x)=k^\star(x)$.}
\end{cases}$$

Therefore, for each $k\in\mathbb{N}_0,$ $\varphi(x,k)=\lim_{l\rightarrow \infty}\varphi_{\gamma_l}^\star(x,k)=\varphi_{\gamma_{l_0}}^\star(x,k).$ Thus,  (\ref{Kostia2026Eqn11}) holds for $k^\star(x)=k_{\gamma_{l_0}}^\star(x)$ according to Theorem \ref{thm:elapsed_time_threshold}. 
\end{proof}

\end{document}